\documentclass{amsart}
\usepackage[utf8]{inputenc}
\usepackage{amssymb,amsmath,bm,enumerate,stmaryrd,amsthm,tikz-cd}
\usepackage{stackrel}
\usetikzlibrary{matrix}
\usetikzlibrary{arrows}
\usepackage[all]{xy}
\usepackage[shortlabels]{enumitem}
\usepackage{colonequals}
\usepackage{eucal}

\makeatletter
\@namedef{subjclassname@2020}{\textup{2020} Mathematics Subject Classification}
\makeatother

\usepackage[colorlinks]{hyperref}
\hypersetup{
 colorlinks=true,
 linkcolor=blue,
 filecolor=magenta, 
 urlcolor=cyan,
}
\usepackage[capitalise]{cleveref}
\crefformat{equation}{(#2#1#3)}
\crefrangeformat{equation}{(#3#1#4--#5#2#6)}
\crefformat{enumi}{(#2#1#3)}
\crefrangeformat{enumi}{(#3#1#4--#5#2#6)}

\usepackage{tikz-cd}
\usetikzlibrary{decorations.pathmorphing}

\makeatletter
 \newcommand{\lra}{\longrightarrow}
\newcommand{\xra}{\xrightarrow}

 \newcommand{\ges}{\geqslant}
\newcommand{\dcat}[2][]{\cat{D}_{#1}(#2)}
\newcommand{\dfcat}[2][]{\cat{D}^{\sf{f}}_{#1}(#2)}
\newcommand{\dfbcat}[1]{\dfcat[\sf{b}]{#1}}

\newcommand{\dflbcat}[1]{\cat{D}^{\sf{fl}}_{\sf{b}}(#1)}

\newcommand{\fm}{\mathfrak{m}}

\newcommand{\fp}{\mathfrak{p}}
\newcommand{\vp}{{\varphi}}

\newcommand{\AQC}[4]{\operatorname{D}^{#1}(#2/#3;#4)}
\newcommand{\atiyah}[2]{\mathsf{at}^{#1}(#2)}
\newcommand{\at}[1]{\mathsf{at}^{#1}}
\newcommand{\cat}[1]{{\mathsf{#1}}}

\newcommand{\cotan}[1]{\operatorname{L}({#1})}
\newcommand{\env}[2][]{{#2}_{#1}^{\operatorname{e}}}
\DeclareMathOperator{\Ext}{Ext}
\newcommand{\ext}[1]{\Ext_\vp^i(#1,#1)}

\newcommand{\builds}[1][]{\mathrel{\mathop{\models}\limits_{\vbox to 0.5em{\kern-3\ex@\hbox{$\scriptstyle\,\,#1$}\vss}}}}
\newcommand{\hh}{\operatorname{H}}
\newcommand{\hc}[3][{*}]{\operatorname{HH}^{#1}(#2/#3)}

\DeclareMathOperator{\Hom}{Hom}
\DeclareMathOperator{\id}{id}
\DeclareMathOperator{\Image}{Im}
\newcommand{\Ker}{\operatorname{Ker}}

\newcommand{\lotimes}{\otimes^{\operatorname{L}}}

\newcommand{\pdim}{\operatorname{proj\,dim}}

\newcommand{\RHom}{\operatorname{RHom}}
\newcommand{\shift}{{\sf\Sigma}}
\newcommand{\supp}{\operatorname{supp}}
\DeclareMathOperator{\thick}{thick}
\DeclareMathOperator{\Thick}{Thick}
\DeclareMathOperator{\Tor}{Tor}
\DeclareMathOperator{\depth}{depth}

\newcommand{\fbuilds}[1][]{\mathrel{\mathop{\models}\limits_{\vbox to 0.5em{\kern-3\ex@\hbox{$\scriptstyle\,\,#1$}\vss}}}}
\makeatother

\newcounter{intro}
\newtheorem{introthm}[intro]{Theorem}

\newtheorem{theorem}[subsection]{Theorem}

\newtheorem{lemma}[subsection]{Lemma}
\newtheorem{corollary}[subsection]{Corollary}

\theoremstyle{definition}
\newtheorem{example}[subsection]{Example}
\newtheorem{chunk}[subsection]{}

\theoremstyle{remark}
\newtheorem{remark}[subsection]{Remark}

\numberwithin{equation}{subsection}

\title[Exceptional complete intersection maps]{Exceptional complete intersection\\ maps of local rings}

\author[S.B.~Iyengar]{Srikanth~B.~Iyengar}
\address{Srikanth~B.~Iyengar,
Department of Mathematics,
University of Utah,
Salt Lake City, UT 84112,
U.S.A.}
\email{iyengar@math.utah.edu}

\author[J.C.~Letz]{Janina~C.~Letz}
\address{Janina~C.~Letz,
Faculty of Mathematics,
Bielefeld University,
PO Box 100 131,
33501 Bielefeld,
Germany}
\email{jletz@math.uni-bielefeld.de}

\author[J.~Liu]{Jian Liu}
\address{Jian Liu,
School of Mathematical Sciences,
Shanghai Jiao Tong University, Shanghai 200240,
P.R. China.}
\email{liuj231@sjtu.edu.cn}

\author[J.~Pollitz]{Josh Pollitz}
\address{Josh Pollitz,
Department of Mathematics,
University of Utah,
Salt Lake City, UT 84112,
U.S.A.}
\email{pollitz@math.utah.edu}

\date{22 February 2022}



\thanks{This work was partly supported by National Science Foundation DMS grants 2001368 (SBI) and 2002173 (JP) and by the Alexander von Humboldt Foundation in the framework of an Alexander von Humboldt Professorship endowed by the German Federal Ministry of Education and Research (JCL)}

\begin{document}

\begin{abstract}
This work concerns surjective maps $\varphi\colon R\to S$ of commutative noetherian local rings with kernel generated by a regular sequence that is part of a minimal generating set for the maximal ideal of $R$. The main result provides criteria for detecting such exceptional complete intersection maps in terms of the lattices of thick subcategories of the derived category of complexes of finite length homology. A key input is a characterization of such maps in terms of the truncated Atiyah class of $\varphi$.
\end{abstract}

\maketitle

%\tableofcontents

%%%%%%%%%%%%%%%%%%%%%%%%%%%%%%%%%%%%%%%%%%%%%%%%%%%%%%%%%%%
%%%%%%%%%%%%%%%%%%%%%%%%%%%%%%%%%%%%%%%%%%%%%%%%%%%%%%%%%%%
\section*{Introduction}
This work is part of an ongoing project to understand properties of commutative noetherian rings, and of maps between them, in terms of the triangulated structure of their derived categories. Building on the work of Dwyer, Greenlees and the first author~\cite{Dwyer/Greenlees/Iyengar:2006b}, the fourth author~\cite{Pollitz:2019} established a homotopical description of locally complete intersection rings: These are the commutative noetherian rings with the property that any nonzero thick subcategory of its bounded derived category contains a nonzero perfect complex. Similar homotopical characterizations of the complete intersection property have since appeared in the works of Briggs, Grifo and the authors~\cite{Briggs/Grifo/Pollitz:2021,Briggs/Iyengar/Letz/Pollitz:2021,Letz:2021}; the most general result is contained in \cite{Briggs/Iyengar/Letz/Pollitz:2021}. 

A surjective map $\vp\colon R\to S$ of local rings is complete intersection if its kernel is generated by an $R$-regular sequence. The focus of the present manuscript is on a special class of complete intersection maps where, in addition, the regular sequence is part of a minimal generating set for the maximal ideal of $R$. Such maps arise naturally in various contexts; for example, the diagonal of a smooth map has this property locally. They also arose in recent work of the first author with Brochard and Khare~\cite{Brochard/Iyengar/Khare:2021}, where they are called \emph{exceptional complete intersection} maps; we adopt this terminology. 

Our motivation for the present undertaking comes from \cite{Briggs/Iyengar/Letz/Pollitz:2021} that raises the prospect of defining locally complete intersection maps in the broader setting of tensor triangulated categories. However to realize this goal requires a homotopical description of the diagonal of a smooth map, and more generally, of a weakly regular map. Our main result provides such a description for general exceptional complete intersection maps; see also the discussion in \cref{dflasatriangulatedcat}. Let $\dflbcat{R}$ be the subcategory of the derived category of $R$ consisting of complexes with finite length homology. The statement involves the functor $\vp_*\colon \dflbcat{S}\to \dflbcat{R}$ induced by restriction along $\vp$. 

\begin{introthm}
\label{introthm2}
Let $\vp\colon R\to S$ be a surjective map of commutative noetherian local rings. Then $\vp$ is exceptional complete intersection if and only if 	$\pdim_R S$ is finite and $\vp_*$ induces an isomorphism between the lattice of thick subcategories of $\dflbcat{S}$ and that of $\dflbcat{R}$.
\end{introthm}

See \cref{cor:latticeiso} for a more precise statement. The exceptional complete intersection property does not localize and so,  unsurprisingly, our characterization is in terms of the complexes supported only at the maximal ideals of $R$ and $S$. One cannot, in general, relax this hypothesis to allow thick subcategories with bigger support; see \cref{exp:Nilpotencydegree} and \cref{NilpotencyAscentDescent}. It is also worth pointing out that the characterizations of the complete intersection property in \cite{Briggs/Grifo/Pollitz:2021,Briggs/Iyengar/Letz/Pollitz:2021,Pollitz:2019} use objects from $\dflbcat{R}$ in an essential way. 

Questions about the thick subcategories of $\dfbcat{R}$, the bounded derived category of $R$, can often be reduced to those about the thick subcategories of $\dflbcat{R_{\fp}}$ as $\fp$ ranges over the prime ideals in $R$; this is the upshot of local-to-global principles from \cite{Benson/Iyengar/Krause:2015}. When the ring $R$ is Cohen--Macaulay, \cref{introthm2} allows one to further reduce to the case of artinian rings. This idea is illustrated in \cref{cor:artinian}.

To each $S$-complex $M$ there is a naturally defined morphism in $\dcat S$ called the \emph{truncated Atiyah class of $\vp$ at $M$}, denoted $\atiyah{\vp}{M}$, which was introduced in \cite{Huybrechts/Thomas:2010}; see \@ \cref{truncatedAtiyah}. The following characterization of exceptional complete intersection maps in terms of the truncated Atiyah class of $k$ is an input in our proof of \cref{introthm2}.

\begin{introthm}
\label{introthm1}
Let $\vp\colon R\to S$ be a surjective map of commutative noetherian local rings, and let $k$ be their common residue field. Then $\vp$ is exceptional complete intersection if and only if $\pdim_R S$ is finite and $\atiyah{\vp}{k}=0$.
\end{introthm}

This result is part of \cref{characterization1}. It is essentially contained in \cite[Proposition~2.8]{Avramov:1989} and \cite[3.11]{Briggs/Iyengar/Letz/Pollitz:2021}, though couched in terms of actions of Hochschild cohomology. Our contribution is to clarify the connection with the Atiyah class; see, especially, \cref{mainm}. 

\cref{introthm1} also sheds light on the faithfulness (or lack thereof) of the restriction functor $\vp_*\colon \dfbcat{S}\to \dfbcat{R}$, at least when $\vp$ is complete intersection. In this case, $\vp_*$ is faithful if and only if $\atiyah{\vp}{M}$ is zero for each $M$ in $\dcat{S}$; see \cref{ch:faithful}. This holds, for example, when $\vp$ is an algebra retract, but not always, even for exceptional complete intersection maps. This is connected to the ``lifting problem"; see \cref{exampleregular}. This connection also makes it clear that in \cref{introthm1} the vanishing of $\atiyah{\vp}{k}$ cannot be replaced with the stronger conclusion that $\atiyah{\vp}{M}$ is zero for each $M$ in $\dcat{S}$; see \cref{exp:Nilpotencydegree,exampleregular}. One can thus view the non-vanishing of the truncated Atiyah class as a homological obstruction to $\vp$ admitting an algebra retract.

%%%%%%%%%%%%%%%%%%%%%%%%%%%%%%%%%%%%%%%%%%%%%%%%%%%%%%%%%%%
%%%%%%%%%%%%%%%%%%%%%%%%%%%%%%%%%%%%%%%%%%%%%%%%%%%%%%%%%%%

\section{Hochschild cohomology}
\label{se:HH}
The Atiyah class of a map is closely connected to Hochschild cohomology of the algebras involved. In this section we record the necessary facts on Hochschild cohomology and their characteristic action on derived categories. Some constructions require dg (=differential graded) algebras and the corresponding derived category of dg modules; everything we need is explained in \cite{Avramov:1998,Avramov/Buchweitz/Iyengar/Miller:2010}. 

\begin{chunk}
We only consider graded-commutative dg algebras, so let $A$ be such a dg algebra. The derived category of dg $A$-modules is denoted $\dcat A$, and its suspension functor is denoted $\shift$. One has the usual bi-functors $\RHom_A(-,-)$ and $-\lotimes_A-$ on $\dcat{A}$; for dg $A$-modules $M$ and $N$, we set
\[
\Ext^*_A(M,N)\colonequals\hh^*(\RHom_A(M,N))\text{ and } \Tor^A_*(M,N)\colonequals\hh_*(M\lotimes_A N)\,.
\]	
An element $\alpha$ in $\Ext_A^n(M,N)$ corresponds to a morphism $M\to \shift^{n} N$ in $\dcat{A}$; by a slight abuse of notation we write also $\alpha$ for this morphism. Given $\alpha\colon M\to \shift^n M$ we write $\alpha^i$ for the composition 
\[
M\xra{\ \alpha\ }\shift^n M\xra{\shift^n\alpha}\ldots \xra{\shift^{n(i-1)}\alpha}\shift^{ni} M
\] 
in $\dcat{A}$, and say $\alpha$ is \emph{nilpotent} if $\alpha^i=0$ in $\dcat A$ for some $i\ges 0$. 
\end{chunk}

%%%Pick up here

\begin{chunk}
\label{ch:derived-enveloping}
Let $\vp\colon R\to S$ be a map of commutative rings. Set $\env[R]S\colonequals S\lotimes_R S$, the (derived) enveloping algebra of $S$ over $R$, and let $\mu_R^S\colon\env[R]S\to S$ be the multiplication map. We can model $\env[R]{S}$ as $S\otimes_R A$ where $\vp$ factors as maps of dg algebras
\[
R\lra A\xra{\ \epsilon\ } S
\] 
with $A$ a dg $R$-algebra such that each $R$-module $A_i$ is flat, and $\epsilon$ is a quasi-isomorphism. The multiplication map $\mu_R^S$ is modeled by $1\otimes\epsilon\colon S\otimes_R A\to S$. This is independent of the choice of factorization of $\vp$ in the following sense: If $R\to A'\to S$ is another such factorization of $\vp$, then $A'\otimes_R S\simeq A\otimes_R S$ as dg algebras augmented to $S$. In particular, $S$ is fixed under the equivalence $\dcat{S\otimes_RA}\equiv\dcat{S\otimes_RA'}$ induced by the aforementioned quasi-isomorphism of dg algebras.
\end{chunk}

We view $S$ as a dg $\env[R]{S}$-module through the multiplication map $\mu^S_R$.

\begin{chunk}
\label{ConstructionHH}
The Hochschild cohomology of $S$ over $R$ is the graded-commutative $S$-algebra
\begin{equation*}
\hc{S}{R} \colonequals \Ext_{\env[R]{S}}^{*}(S,S)\,.
\end{equation*} 
For each $M$ in $\dcat{S}$ there is a morphism of graded $S$-algebras
\[
\chi_M \colon \hc{S}{R} \lra \Ext_S^*(M,M)\,
\]
induced by the functor $- \lotimes_S M \colon \dcat{\env{S}} \to \dcat{S}$. We call $\chi_M$ the \emph{characteristic map of $M$}. It defines a central action of $\hc{S}{R}$ on $\dcat{S}$: for $N$ in $\dcat S$ and maps $\alpha$ in $\hc{S}{R}$ and $\beta$ in $\Ext_S(M,N)$ one has 
\[
\chi_N(\alpha) \beta = (-1)^{|\alpha||\beta|} \beta \chi_M(\alpha) \,.
\] 
\end{chunk}

\begin{chunk}
A subset $U$ of $\hc{S}{R}$ acts \emph{trivially}, respectively, \emph{nilpotently}, on $M$ if for each $\alpha\in U$, the element $\chi_M(\alpha)$ in $\Ext_S(M,M)$ is zero, respectively, nilpotent.  For example, $\hc[\geqslant 1]{S}{R}$ acts trivially on the $S$-complex $S\lotimes_R L$ for any $R$-complex $L$. The subcategory of $\dcat S$ consisting of $S$-complexes on which $U$ acts nilpotently is thick; that is to say, it is a triangulated subcategory of $\dcat S$ closed under direct summands. These observation will be used often in what follows.
\end{chunk}

\section{The truncated Atiyah class}
\label{se:tac}
In this section we recall the construction of the truncated Atiyah class, introduced by Huybrechts and Thomas in \cite{Huybrechts/Thomas:2010}, and discuss its connection to Hochschild cohomology. While the truncated Atiyah class, like the Atiyah class itself, can be defined for any map of rings, we only need it for surjective maps.

Throughout $\vp\colon R\to S$ is a surjective map of commutative noetherian rings, with kernel $I$. The discussion below is extracted from \cite{Briggs/Iyengar/Letz/Pollitz:2021}.

\begin{chunk}
\label{truncatedAtiyah}
The multiplication map $\mu^S_R$ from \cref{ch:derived-enveloping} extends in an exact triangle 
\[
 \env[R]S \xra{\ \mu\ } S \to \shift J\to 
\]
in $\dcat{\env[R]{S}}$. From the triangle above, it is clear that 
\[
\hh_i(J)=
\begin{cases}
\Tor_{\ges 1}^R(S,S) & i\geqslant 1;\\
0 & i\leqslant 0\,.
\end{cases}
\]
Thus the canonical truncation yields a morphism $J\to \shift I/I^2$ in $\dcat{\env[R]S}$. The \emph{truncated Atiyah class} of $\vp$, denoted $\at{\vp}$, is the composition of morphisms
\[
S\to \shift J\to \shift^2I/I^2
\] 
in $\dcat{\env[R]{S}}$. Consider the map sending $\alpha$ in $\Hom_S(I/I^2,S)$ to the composition
\begin{equation*}
\label{truncatedatiyahdiagram}
\begin{tikzcd}
S \arrow[r,"\at{\vp}"] & \shift^2 I/I^2\arrow[r,"\shift^2\alpha"] & \shift^2 S
\end{tikzcd}
\end{equation*}
in $\dcat{\env[R]{S}}$. By \cite[Lemma~2.8]{Briggs/Iyengar/Letz/Pollitz:2021}, this is an isomorphism 
\[
\delta^\vp\colon \Hom_S(I/I^2,S)\xra{\cong} \hc[2]{S}{R}\,.
\]
For each $M$ in $\dcat{S}$ and $\alpha$ in $\Hom_S(I/I^2,S)$, applying $-\lotimes_S M$ yields the following commutative diagram in $\dcat{S}$
\begin{equation}
\label{hh2atiyahdiagram}
\begin{tikzcd}
M \arrow[r,"\at{\vp}\lotimes_S M"]\arrow[rd,swap,"\chi_M(\delta^\vp(\alpha))"]&[2em] \shift^2 I/I^2\lotimes_S M\arrow[d,"\shift^2\alpha\otimes 1"] \\
& \shift^2 M\,.
\end{tikzcd}
\end{equation}
The horizontal map in this diagram is the \emph{truncated Atiyah class of $\vp$ at M}, denoted $\atiyah \vp M$. Thus we can interpret $\atiyah \vp M$ as a class in $\Ext_S^2(M, I/I^2\lotimes_SM)$. 
\end{chunk}

\begin{chunk}
\label{fullAtiyah}
As discussed in \cite{Huybrechts/Thomas:2010}, the truncated Atiyah class $\atiyah{\vp}{M}$ can be recovered from the universal Atiyah class 
\[
{\sf At}^\vp\colon S\to \shift {\rm L}(\vp)
\] 
where ${\rm L}(\vp)$ is the cotangent complex of $\vp$; cf.\@ \cite{Illusie:1971,Quillen:1969}.
Namely, $\atiyah{\vp}{M}$ factors as 
\[
M\xra{{\sf At}^\vp\lotimes_S M}\shift{\rm L}(\vp)\lotimes_S M\to \shift^2I/I^2\lotimes_S M
\] 
where the second map is induced by the soft truncation using that $\hh_0({\rm L}(\vp))=0$ and $\hh_1({\rm L}(\vp))= I/I^2$. 
The cotangent complex and universal Atiyah class of $\vp$ are defined via simplicial resolutions and since this machinery is not needed in this article we direct the interested reader to \emph{loc.\@ cit.\@} for further details. 
\end{chunk}

\begin{chunk}
\label{atiyahforcis}
Suppose now that the surjective map $\vp\colon R\to S$ is complete intersection of codimension $c$; that is to say, that the ideal $I$ is generated by an $R$-regular sequence of length $c$. Then the $S$-module $I/I^2$ is a free $S$-module of rank $c$, and the map $\delta^\vp$ induces the isomorphism of $S$-algebras
\[
\mathrm{Sym}(\delta^\vp)\colon \mathrm{Sym}_S(\Hom_S(I/I^2,S)) \xra{\ \cong \ } \hc{S}{R}
\]
where $\Hom_S(I/I^2,S)$ sits in cohomological degree 2. See \cite[Section~3]{Avramov/Buchweitz:2000a} for a proof. Choosing a basis for the $S$-module $I/I^2$ gives rise to a description of $\hc{S}{R}$ as a polynomial ring over $S$, in $c$ variables of cohomological degree $2$. These are the operators of Gulliksen \cite{Gulliksen:1974} and Eisenbud \cite{Eisenbud:1980}; cf.\@ \cite{Avramov/Sun:1998}. 
\end{chunk}

\begin{lemma}
\label{atiyahhh2ci}
Let $\vp\colon R\to S$ be a surjective complete intersection map, and $M$ an $S$-complex. One has $\atiyah \vp M =0$ if and only if $\hc[2]{S}{R}$ acts trivially on $M$. 
\end{lemma}
\begin{proof}
Set $I\colonequals \Ker \vp$ and consider the composition in $\dcat{S}$
\[
M\xra{\atiyah{\vp}{M}} \shift^2I/I^2\lotimes_S M\simeq \shift^2I/I^2\otimes_S M\,.
\]
The desired result follows from this, \cref{hh2atiyahdiagram}, and the freeness of $I/I^2$ over $S$.
\end{proof}

\begin{chunk}\label{c:triangle}
Assume $\vp$ is complete intersection of codimension one. By \cref{atiyahforcis}, $\hc[2]{S}{R}$ is isomorphic to $S$, and so choosing a basis element $\eta$ for the free $S$-module $\hc[2]{S}{R}$ identifies $\atiyah{\vp}{M}$ with $\chi_M(\eta)$ in $\Ext_S^2(M,M)$. Finally, from the triangle defining the Atiyah class one obtains the exact triangle in $\dcat S$ 
\[
M\xra{\chi_M(\eta)} \shift^2 M\to \shift S\lotimes_R \vp_*(M)\to 
\] 
 where $\vp_*\colon \dcat{S}\to \dcat{R}$ is the restriction functor; see, for example, \cite{Avramov:1989}.
\end{chunk}

The arguments in the next section use some basic results on Andr\'{e}--Quillen cohomology and homotopy Lie algebras of local rings. We recall the relevant points below; for further details see \cite{Iyengar:2007,Majadas/Rodicio:2010} and \cite[Section~10]{Avramov:1998}.

\begin{chunk}
\label{AQfacts}
Let $\vp\colon R\to S$ be a surjective map of local rings, and let $k$ be their common residue field. We write $\AQC iSRk$ for the $i$th Andr\'{e}--Quillen cohomology module of $\vp$ with coefficients in $k$. We need elementary properties whose proofs are contained in the references above. 

First, one has the following natural identification
\[
\AQC 1SRk \cong \Hom_k(I/\fm_R I,k)
\]
where $\fm_R$ is the maximal ideal of $R$ and $I$ is the kernel of $\vp$. Second, there is a canonically induced exact sequence of $k$-spaces 
\[
0\lra \AQC 1kSk\lra \AQC 1kRk\lra \AQC 1SRk\xra{\eth} \AQC 2kSk\lra \ldots 
\]
called the Jacobi--Zariski sequence corresponding to the composition $R\xra{\vp} S\to k$. 
\end{chunk}

\begin{chunk}
\label{piexactsequence} 
Fix a local ring $R$ with residue field $k$. Let $\pi(R)$ denote the homotopy Lie algebra of $R$. It is a graded Lie algebra over $k$ whose universal enveloping algebra is $\Ext_R(k,k)$. By the Poincar\'{e}--Birkhoff--Witt theorem, a $k$-basis for $\Ext_R(k,k)$ consists of all compositions of the form 
 \[
 \zeta_1^{i_1}\zeta_2^{i_2}\ldots \zeta_n^{i_n}
 \]
 where $\zeta_j$ is in $\pi(R)$ with $i_j\geqslant 0$ when $|\zeta_j|$ is even and $i_j=0,1$ when $|\zeta_j|$ is odd; see, for example, \cite[Theorem~10.2.1]{Avramov:1998}. In particular, a nonzero element in $\pi(R)$ of even degree is not nilpotent when regarded as an element of $\Ext_R(k,k)$.

Let $\vp\colon R\to S$ be a surjective map of local rings. There is a naturally defined map of graded Lie algebras 
\[
\pi(\vp)\colon \pi(S)\to \pi(R)\,.
\]
An important point is that for $i=1,2$, the map $\vp$ determines a commutative diagram 
 \begin{equation}\label{comparisonmaps}
 \begin{tikzcd}
 	\AQC ikSk \ar[r,"\cong"] \ar[d,swap,"\AQC ik\vp k"]&\pi^i(S)\ar[d,"\pi^i(\vp)"]\ar[r,hookrightarrow]& \Ext_S^i(k,k)\, \ \ar[d,"\ext{k}"] \\
 	\AQC ikRk \ar[r,"\cong"] &\pi^i(R) \ar[r,hookrightarrow]& \Ext_R^i(k,k)\,.
 \end{tikzcd}
 \end{equation} 
Consider the diagram of $S$-modules
\begin{equation} \label{maindiagram}
\begin{tikzcd}
	\hc[2]{S}{R}\ar[dd,swap,"\chi_k"] &\Hom_S(I/I^2,S) \ar[l,swap,"\delta^\vp"] \ar[l,"\cong"] \ar[r] & \Hom_S(I/I^2,k)\ar[d,"\cong"]\ar[ddll,"\psi"]\\ 
 & & \AQC 1 SRk \ar[d,"\eth"]	\\
	\Ext_S^2(k,k)&\ar[l,"\iota"] \pi^2(S) & \ar[l,"\cong"]\AQC 2 kSk \, 
\end{tikzcd}
\end{equation}	
where $\iota$ is the canonical inclusion and the unlabelled horizontal arrow is induced by $S\to k$; the map $\psi$ is obtained by applying $\Hom_{\dcat{S}}(-,\shift^2k)$ to the composition 
\[
k\xra{\atiyah{\vp}{k}}\shift^2I/I^2\lotimes_S k\to \shift^2 I/I^2\otimes_S k
\]
 where the second map is the augmentation.
 \end{chunk}

\begin{lemma}\label{mainm}
Diagram \cref{maindiagram} is commutative, and  $\Image(\psi)\supseteq\chi_k(\hc[2]{S}{R})$ with equality when $\vp$ is complete intersection. 
\end{lemma}
\begin{proof}
By \cite[3.11]{Briggs/Iyengar/Letz/Pollitz:2021}, the image of $\psi$ is $\Ker\pi^2(\vp)$. Now the commutativity of \cref{maindiagram}, and hence the desired containment,  follows from the observations in \cref{AQfacts} and \cref{comparisonmaps}.  When $\vp$ is complete intersection, the containment is an equality since $I/I^2$ is a free $S$-module. 
\end{proof}

%%%%%%%%%%%%%%%%%%%%%%%%%%%%%%%%%%%%%%%%%%%%%%%%%%%%%%%%%%%
\section{Exceptional complete intersections}
\label{se:ec}
In this section we prove \cref{introthm1} from the introduction, and present examples that witness some, perhaps, surprising behavior.

\begin{chunk}
\label{eci}
Let $\vp\colon R\to S$ be a surjective map of noetherian local rings; set $I\colonequals \Ker \vp$. Following \cite{Brochard/Iyengar/Khare:2021} we say $\vp$ is an \emph{exceptional complete intersection} provided $I$ is generated by a regular sequence that can be extended to a minimal generating set for $\fm_R$, the maximal ideal of $R$. The latter condition means that the natural map
\[
\frac{I}{\fm_R I}\lra \frac{\fm_R}{\fm_R^2}
\]
is one-to-one. We refer to \cite{Brochard/Iyengar/Khare:2021}, where this notion is introduced for not necessarily surjective maps of local rings, for various characterizations of exceptional complete intersection maps in terms of numerical invariants associated to $R$ and $S$. Here is one more, extracted from the proof of \cite[Theorem~3.1]{Brochard/Iyengar/Khare:2021}.

\begin{lemma}
\label{eci-fpd}
Let $\vp\colon R\to S$ be a surjective map of noetherian local rings. The map $\vp$ is exceptional complete intersection if and only if $\pdim_RS$ is finite and the natural map $I/\fm_R I \to \fm_R/\fm_R^2$ is one-to-one. 
\end{lemma}

\begin{proof}
Since a complete intersection map has finite projective dimension, the only if direction is clear. As to the converse: We can assume $I\ne 0$. Then the hypothesis that $\pdim_RS$ is finite means that $I$ contains a nonzero divisor, say $x$; this is by a result of Auslander and Buchsbaum~\cite[Corollary 1.4.7]{Bruns/Herzog:1998}. Since $I\not\subseteq \fm^2_R$, one can ensure that $x$ is not in $\fm^2_R$; this is by a variation of the prime avoidance argument~\cite[Lemma~1.2.2]{Bruns/Herzog:1998}. Setting $R'\colonequals R/xR$, the map $\vp$ factors as
\[
R\lra R'\lra S\,.
\] 
By the choice of $x$, the map $R\to R'$ is exceptional complete intersection. Thus since $\pdim_RS$ is finite so is $\pdim_{R'}S$; this is by Nagata's theorem recalled in \ref{ch:Nagata}. Moreover one can verify easily that the hypothesis on $I$ is inherited by the kernel of the map $R'\to S$. We can thus induce on the number of generators of $I$ to deduce that $\vp$ is exceptional complete intersection. 
\end{proof}
\end{chunk}

The result below, due to Nagata~\cite[Corollary~27.5]{Nagata:1962}, was used above; this will be generalized later in \cref{thm:nagata+}.

\begin{chunk}
\label{ch:Nagata}
When $\vp\colon R\to S$ is exceptional complete intersection and $M$ is a finitely generated $S$-module, $\pdim_R M$ is finite if and only if $\pdim_S M$ is finite.
\end{chunk}

The following result contains \cref{introthm1} from the introduction. Most of the groundwork for the proof has been laid in \cref{se:tac}.

\begin{theorem}
\label{characterization1}
Let $\vp\colon R\to S$ be a surjective map of noetherian local rings with $\pdim_R S$ finite. The conditions below are equivalent: 
	\begin{enumerate}
	\item 
	\label{characterization1:eci}$\vp$ is exceptional complete intersection;
	\item \label{characterization1:trivial}
	$\vp$ is complete intersection and $\hc[2]{S}{R}$ acts trivially on $k$;
	\item \label{characterization1:nilpotence} $\vp$ is complete intersection and $\hc[2]{S}{R}$ acts nilpotently on $k$;
		\item \label{characterization1:vanish} $\atiyah\vp k=0$, where $k$ is the residue field of $S$.
	\end{enumerate}
\end{theorem}

\begin{proof}
\cref{characterization1:eci}$\Rightarrow$\cref{characterization1:trivial}: 
From the assumption and \cite[Proposition~2.8]{Avramov:1989} it follows that $\hc[2]{S}{R}$ acts trivially on $k$.

\cref{characterization1:trivial}$\Leftrightarrow$\cref{characterization1:nilpotence}: This follows from \cref{mainm}, since nonzero elements of $\pi^2(S)$ cannot be nilpotent when regarded as maps in $\Ext_S^2(k,k);$ cf.\@ \cref{piexactsequence}. 

\cref{characterization1:trivial}$\Rightarrow$\cref{characterization1:vanish}: 
This is contained in \cref{atiyahhh2ci}.
 
 \cref{characterization1:vanish}$\Rightarrow$\cref{characterization1:eci}: Given that $\pdim_RS$ is finite and \cref{eci-fpd} it suffices to prove that the natural map $I/\fm_R I\to \fm_R/\fm_R^2$ is one-to-one; equivalently, that its $k$-vector space dual is onto. By assumption the commutative diagram \cref{maindiagram} yields $\eth=0$ in the Jacobi--Zariski sequence corresponding to $R\to S\to k$ from \ref{AQfacts}. Thus the latter reduces to an exact sequence
 \[
 0\lra \left(\frac{\fm_S}{\fm_S^2}\right)^\vee\lra \left(\frac{\fm_R}{\fm_R^2}\right)^\vee\lra \left(\frac{I}{\fm_R I}\right)^\vee\lra 0
 \]
where $(-)^\vee$ denotes the vector space dual and the maps are the canonically induced ones. This is the desired conclusion. 
\end{proof}

\begin{chunk}
Recall from \cref{fullAtiyah} that in $\dcat{S}$ there is a commutative diagram 
\[
\begin{tikzcd}[column sep=large]
	k \ar[r,"{\sf{At}}^\vp\otimes_S k"] \ar[d,equal]& \shift \cotan{\vp}\otimes_S k\ar[d]\\
	k \ar[r,swap,"\atiyah{\vp}{k}"] & \shift^2 I/I^2\lotimes_S k
\end{tikzcd}
\]
where the map on the right is induced by the soft truncation $\tau\colon\cotan{\vp}\to \shift I/I^2$; it is an isomorphism when $\vp$ is complete intersection. It follows from this observation and \cref{characterization1} that $\vp$ is exceptional complete intersection if and only if $\pdim_R S$ is finite and the map ${\sf{At}}^\vp\otimes_S k$ is zero in $\dcat{S}$. 
\end{chunk}

\begin{chunk}
\label{nilpotencefinitelength}
Let $\dflbcat{S}$ denote the full subcategory of $\dcat{S}$ consisting of $S$-complexes $M$ such that the $S$-module $\hh(M)$ has finite length. Since $\dflbcat{S}$ is the smallest thick subcategory of $\dcat S$ containing $k$---see, for example, \cite[3.5]{Dwyer/Greenlees/Iyengar:2006b}---it follows from \cref{characterization1} that when $\vp$ is exceptional complete intersection, the action of $\hc[2]{S}{R}$ on $\dflbcat{S}$ is nilpotent; however there may be no bound on the nilpotence degree. Moreover $\hc[2]{S}{R}$ need not act nilpotently on objects not in $\dflbcat{S}$. The example below illustrates these points.
 \end{chunk}

\begin{example} \label{exp:Nilpotencydegree}
Let $k$ be a field, consider the hypersurface ring
\[
R \colonequals \frac{k\llbracket x,y,z\rrbracket}{(x^2-yz)}
\]
and the exceptional complete intersection map 
\[
\vp\colon R\to S \colonequals R/(z)\cong \frac{k\llbracket x,y\rrbracket}{(x^2)}\,.
\] 
Let $\eta$ denote the class dual to the class of $z$ in $(z)/(z^2)$, viewed as an element of $\hc[2]{S}{R}$ via the isomorphism $\delta^\vp$. 

Set $M\colonequals S/(x)$, regarded as an $S$-module via the surjection $S\to S/(x)$. The minimal resolution of $M$ over $S$ is $\ldots \xra{x} S \xra{x} S\to 0$. From there it is clear that $\Ext_S^i(M,M)=M$ for all $i\ges 0$. Moreover using the construction of cohomology operators from \cite[Section~1]{Eisenbud:1980}, one gets that multiplication by $\chi_M(\eta)$ on $\Ext_S(M,M)$ is represented by 
\[
M\xra{y\cdot} M\,.
\]
In particular $\hc[2]{S}{R}$ does not act nilpotently on $M$. 

In the same vein for $N\colonequals S/(x,y^n)$, where $n$ is any positive integer, one gets
\[
\Ext_S^i(N,N)=
\begin{cases} 
N & i=0 \\ 
N^{\oplus 2}	 &i>0 
\end{cases}
\] 
and the action of $\chi_N(\eta)$ is represented by
\[
N^{\oplus 2} 
\xra{\begin{pmatrix} 	
	y & 0\\
	0 & y
\end{pmatrix}
}N^{\oplus 2}.
\]
Thus the $\hc[i]{S}{R}$ action on $N$ is nontrivial for $0\leqslant i<2n$ and trivial for $i\geqslant 2n$. 
\end{example}

\begin{chunk}\label{ch:faithful}
Let $\vp \colon R \to S$ be complete intersection. Consider the restriction functor
\begin{equation*}
\vp_* \colon \dcat{S} \to \dcat{R}\,.
\end{equation*}
It is straightforward from \cref{atiyahhh2ci,c:triangle} to see the following are equivalent: 
\begin{enumerate}
\item $\vp_*$ is faithful on morphisms;
\item \label{ch:trivial}$\hc[2]{S}{R}$ acts trivially on each object of $\dcat{S}$;
\item $\atiyah \vp M $ is zero for each object $M$ in $\dcat{S}$.
\end{enumerate}
When $\vp$ admits an algebra section---that is to say, a map of rings $\sigma\colon S\to R$ such that $\vp\sigma=\id^S$---the restriction functor $\vp_*$ is faithful and so $\atiyah \vp M$ vanishes for each $M$ in $\dcat{S}$. 

Contrast this with the exceptional complete intersection map in \cref{exp:Nilpotencydegree} where a class of modules demonstrate nonzero, even non-nilpotent, truncated Atiyah class. Hence, when $\vp$ is exceptional complete intersection, non-vanishing of truncated Atiyah classes serve as homological obstructions to the regular sequence generating $\Ker\vp$ being, loosely speaking, independent power series variables in $R$. These can also be viewed as obstructions to the notion of weak liftability discussed in \cite{Auslander/Ding/Solberg:1993,Yoshino:1997}. 
\end{chunk}

\begin{example}
\label{exampleregular}
Assume $\vp\colon R\to S$ is exceptional complete intersection with $R$, and hence $S$, a complete regular local ring. Whether $\vp_*$ is faithful depends on the characteristic of $R$. When $R$ is equicharacteristic, $\vp$ admits an algebra section; hence, by \cref{ch:faithful} we conclude that $\vp_*$ is faithful. 

In contrast, in \cite[Example~1]{Hochster:1975}, Hochster presented an example which settled Grothendieck's lifting problem in the negative. By combining a calculation of Dao \cite[Example~3.5]{Dao:2007} and a result of Yoshino \cite[Theorem~3.11]{Yoshino:1997},
the module considered by Hochster has non-vanishing truncated Atiyah class over an exceptional complete intersection map whose base ring is a mixed characteristic regular local ring.
\end{example}

\section{Homological invariants}
In this section we track the change of certain homological invariants along exceptional complete intersection maps. The results obtained suggest that, notwithstanding the remarks in \ref{ch:faithful} and \ref{exampleregular}, homological properties of the source and target of an exceptional complete intersection map are tightly linked.

As before, given a local ring $S$, we write $\dfbcat{S}$ for the full subcategory of $\dcat{S}$ consisting of objects with finitely generated homology. 

\begin{theorem}
\label{thm:nagata+}
	Let $\vp \colon R\to S$ be exceptional complete intersection and $M,N$ in $\dfbcat{S}$. Then the following hold:
	\begin{enumerate}
	\item 	\label{extvanishing}$\Ext_R^{i}(M,N)=0$ for $i\gg 0$ if and only if $\Ext_S^i(M,N)=0$ for $i\gg 0$;
	\item \label{torvanishing}$\Tor^R_{i}(M,N)=0$ for $i\gg 0$ if and only if $\Tor^S_i(M,N)=0$ for $i\gg 0$.
	\end{enumerate}
\end{theorem}

\begin{proof}
We first reduce to the case when $\hh(M)$ has finite length: Let $K$ be the Koszul complex on a system of parameters for $S$ and consider the $S$-complex $K\otimes_SM$. Then the length of the $S$-module $\hh(K\otimes_SM)$ is finite. Moreover, since $K\otimes_SM$ can be realized as an iterated mapping cone of maps defined by multiplication by an element of $S$, a straightforward argument yields that 
\begin{align*}
\Ext_R^i(M,N)=0 \quad\text{for $i\gg 0$} \iff \Ext_R^i(K\otimes_SM,N)=0 \quad\text{for $i\gg 0$} \\
\Ext_S^i(M,N)=0 \quad\text{for $i\gg 0$} \iff \Ext_S^i(K\otimes_SM,N)=0 \quad\text{for $i\gg 0$;}
\end{align*}
see, for instance, \cite[2.9]{Dwyer/Greenlees/Iyengar:2006b}. Thus replacing $M$ by $K\otimes_SM$ we can assume the length of $\hh(M)$ is finite. 

(1) Since the length of $\hh(M)$ is finite, the $\hc SR$-action on $\Ext_S(M,M)$, and hence also on $\Ext_S(M,N)$, is nilpotent; see \ref{nilpotencefinitelength}. The desired equivalence is now immediate from \cite[Theorem~3.1]{Gulliksen:1974} and \cite[Theorem~4.2]{Avramov/Gasharov/Peeva:1997} that state the $R$-module $\Ext_R(M,N)$ is finitely generated if and only if the $\hc SR$-module $\Ext_S(M,N)$ is finitely generated.

(2) With $(-)^{\vee}$ denoting Matlis duality of $S$-modules, one has isomorphisms
\[
\Tor^R_i(M,N)^{\vee} \cong \Ext_R^i(M^{\vee},N) \quad\text{and}\quad \Tor^S_i(M,N)^{\vee} \cong \Ext_S^i(M^{\vee},N)\,.
\]
Since $\hh(M)$ has finite length, so does $\hh(M^{\vee})$; the desired equivalence now follows from (1) and the faithfulness of Matlis duality.
\end{proof}

Setting $N\colonequals k$ in either part of \cref{thm:nagata+} recovers Nagata's theorem~\ref{ch:Nagata}, whilst setting $M\colonequals k$ in \cref{{thm:nagata+}}\cref{extvanishing} yields an analogous result for injective dimension. Here is a more precise statement; see \cite{Avramov/Foxby:1991} for the definitions of the projective dimension and injective dimension of a complex.

\begin{corollary}
\label{nagata}
If $\vp \colon R\to S$ is exceptional complete intersection, then for each $M$ in $\dfbcat{S}$ there are equalities
\begin{align*}
\pdim_RM &= \pdim_SM +\dim R-\dim S \\
\mathrm{inj\,dim}_RM & = \mathrm{inj\,dim}_SM + \dim R-\dim S \,.
\end{align*}
\end{corollary}

\begin{proof}
As noted above, from \cref{thm:nagata+} we know $\pdim_RM$ is finite if and only if $\pdim_SM$ is; so assume both values are finite. Now we can apply a generalization of the Auslander--Buchsbaum formula to complexes in \cite[Theorem~II]{Foxby/Iyengar:2003} to obtain the first and third equalities below
\begin{align*}
 \pdim_RM &= \depth R -\depth_RM \\
 &= \depth R-\depth S+\depth S - \depth_SM \\
 &= \depth R-\depth S+ \pdim_SM\\
&= \dim R - \dim S + \pdim_SM\,.
\end{align*}
The last equality holds because $\vp$ is complete intersection.

For the equality involving injective dimensions, from  \cref{thm:nagata+}, we can assume $ \mathrm{inj\,dim}_RM$ and $ \mathrm{inj\,dim}_SM$ are both finite. Now observe that 
\begin{align*}
    \mathrm{inj\,dim}_RM &= \depth R - \inf\{i\in \mathbb{Z}: \hh_i(M)\neq 0\}\\
   &=\depth R-\depth S+ \depth S - \inf\{i\in \mathbb{Z}: \hh_i(M)\neq 0\}\\
   &=\depth R -\depth S+ \mathrm{inj\,dim}_SM\\
   &=\dim R - \dim S +  \mathrm{inj\,dim}_SM
\end{align*}
where the first and third equalities hold by \cite[Lemma~1.4]{Avramov/Iyengar/Lipman:2010}.
\end{proof}

%%%%%%%%%%%%%%%%%%%%%%%%%%%%%%%%%%%%%%%%%%%%%%%%%%%%%%%%%%%
%%%%%%%%%%%%%%%%%%%%%%%%%%%%%%%%%%%%%%%%%%%%%%%%%%%%%%%%%%%
\section{Homotopical characterization}
\label{se:homotopical}
In this section $\vp\colon R\to S$ is a surjective map of commutative noetherian local rings with (common) residue field $k$. Set $I\colonequals \Ker \vp$. The main result in this section is \cref{characterization2}. An essential point is that when $\vp$ is complete intersection information regarding finite building, in the sense recalled below, is related to nilpotence of the action of $\hc[2]{S}{R}$.

 Given an $S$-complex $M$ we write $\thick_S(M)$ for the smallest thick subcategory of $\dcat{S}$ containing $M$; see \cite[Section~2]{Bondal/VanDenBergh:2003}, or \cite[Section~2]{Avramov/Buchweitz/Iyengar/Miller:2010}. Any $S$-complex $N$ in $\thick_S(M)$ is said to be \emph{finitely built} from $M$; we write $M\models_S N$ to indicate that this is so.
 
 Hopkins~\cite{Hopkins:1987} and Neeman~\cite{Neeman:1992b} proved that if $M$ and $N$ in $\dfbcat{S}$ are $S$-complexes of finite projective dimension, and  $\supp_S \hh(M) \supseteq \supp_S \hh(N)$, then $M\models_SN$. Here $\supp_S(-)$ denotes support. This result has the following consequence.

\begin{chunk}
\label{ch:HN}
Let $M$ in $\dfbcat{S}$ be an $S$-complex of finite projective dimension and $\hh(M)\ne 0$.  If $K$ is the Koszul complex on finite set of elements that generate an $\fm_S$-primary ideal, then $M\models_SK$.
\end{chunk}

\begin{lemma}
\label{le:Rbuilding}
If $\pdim_R S$ is finite, then $\thick_R(M)= \thick_R(S\lotimes_RM)$ for any $M$ in $\dcat S$.
\end{lemma}

\begin{proof}
The desired result translates to $M\models_R S\lotimes_R M$ and $S\lotimes_R M\models_R M$. 

As $\pdim_RS$ is finite, $R\models_RS$, and so applying $-\lotimes_RM$ yields $M\models_R S\lotimes_RM$. 

Let $K$ be the Koszul complex on a generating set for the ideal $I$. As $\pdim_RS$ is finite it follows that $S\models_R K;$ see \ref{ch:HN}.
Applying $-\lotimes_R M$ one sees that 
\[
S\lotimes_R M\fbuilds[R] K\otimes_R M\,.
\] 
It remains to observe that since $R$ acts on $M$ through $S$, in $\dcat R$ the complex $K\otimes_R M$ is isomorphic to a finite direct sum of suspensions of $M$. 
 \end{proof}

Our interest is in the statement corresponding to \cref{le:Rbuilding} in $\dcat S$. Let $M$ be an object of $\dcat{S}$. In what follows $S \lotimes_R M$ is always regarded as an $S$-complex via the left $S$-action. That is, $S\lotimes_R M$ denotes $S\lotimes_R \vp_*(M)$ as an object of $\dcat{S}$. 

By \cite[Theorem~8.3]{Dwyer/Greenlees/Iyengar:2006b}, when $\vp\colon R\to S$ is a surjective complete intersection homomorphism, 
\begin{equation}
\label{ThickRSotimesX} 
M\fbuilds[S] S\lotimes_R M
\end{equation} 
but the reverse building is not guaranteed. The obstruction is the non-nilpotence of the action of $\hc[2]{S}{R}$ on $M$; this is the content of the result below.

 \begin{lemma} 
 \label{NilpotencyAscentDescent}
Let $\vp \colon R \to S$ be a surjective complete intersection map. For an $S$-complex $N$ the following conditions are equivalent:
\begin{enumerate}
\item\label{NilpotencyAscentDescent:Nil} $\hc[2]{S}{R}$ acts nilpotently on $N$; 
\item\label{NilpotencyAscentDescent:Thick} $S \lotimes_R N$ finitely builds $N$ in $\dcat{S}$;
\item\label{NilpotencyAscentDescent:AscDes} for any $S$-complex $M$ one has $M\models_R N \iff M\models_S N$. 
\end{enumerate}
\end{lemma}

\begin{proof}
\cref{NilpotencyAscentDescent:Nil} $\Rightarrow$ \cref{NilpotencyAscentDescent:Thick}: 
As $\hc[2]{S}{R}$ acts nilpotently on $N$, any factorization of $\vp$  into a composition of complete intersection maps
\[
R\to R'\to S
\] 
has the property that $\hc[2]{R'}{R}$ and $\hc[2]{S}{R'}$  act nilpotently on $N$ in $\dcat{R'}$ and $\dcat{S},$ respectively; this follows from \cite[3.1]{Avramov/Sun:1998}, or \cite[Proposition~2.3]{Avramov:1989}. 
 Furthermore there is an isomorphism
\[
S\lotimes_{R'}(R'\lotimes_{R}N)\simeq S\lotimes_R N
\] 
in $\dcat{S}.$ Hence, by induction on the codimension of $\vp$, it is enough to show \cref{NilpotencyAscentDescent:Thick} holds when $\vp$ is complete intersection of codimension one. 

In this case, there is an exact triangle
\[
N\xra{\chi_N(\eta)} \shift^2 N\to \shift S\lotimes_R N\to 
\] 
where $\eta$ is a basis element for the free $S$-module $\hc[2]{S}{R}$; cf.\@ \cref{c:triangle}. Iteratively applying the octahedral axiom yields \begin{equation}\label{Building}
S\lotimes_R N\builds[S] {\rm cone} ( \chi_N(\eta)^n)
\end{equation} for any $n\ges 1$. The assumption that $\hc[2]{S}{R}$ acts nilpotently on $N$ means there is some $n\ges 1$ such that $\chi_N(\eta)^n=0$ and so
\[
{\rm cone} ( \chi_N(\eta)^n)\simeq \shift N\oplus \shift^{2n}N
\] 
in $\dcat{S}$. Combining this with \cref{Building} we obtain the desired result.

\cref{NilpotencyAscentDescent:Thick} $\Rightarrow$ \cref{NilpotencyAscentDescent:AscDes}: Evidently, if $M \models_S N$ then $M \models_R N$. When $M \models_R N$, by applying the functor $S\lotimes_R -$ we obtain
\begin{equation*}
 S \lotimes_R M\builds[S] S \lotimes_R N \,;
\end{equation*} from this and the assumption in \cref{NilpotencyAscentDescent:Thick}, we obtain that $S\lotimes_R M\models_S N$. Finally
\[
M\builds[S] S\lotimes_R M
\] 
by \cref{ThickRSotimesX}, so combining this with the already established building $S\lotimes_R M\models_S N$, we conclude that $M\models_S N$ as desired. 

\cref{NilpotencyAscentDescent:AscDes} $\Rightarrow$ \cref{NilpotencyAscentDescent:Nil}: 
From \cref{le:Rbuilding} and the hypothesis in \cref{NilpotencyAscentDescent:AscDes} we get $S\lotimes_R N\models_S N$. 
As discussed in \cref{ConstructionHH} the full subcategory of $\dcat{S}$ consisting of objects for which $\hc[2]{S}{R}$ is nilpotent is a thick subcategory of $\dcat{S}.$ In particular, this subcategory is closed under finite building. Note that $\hc[2]{S}{R}$ has trivial action on $S\lotimes_R N$ and so $\hc[2]{S}{R}$ acts nilpotently on $N$. 
\end{proof}

\begin{chunk}
Let $\vp_*\colon \dcat{S}\to \dcat{R}$ be the restriction functor and $\cat{S}$ a subcategory of $\dcat{S}$. We say that $\vp_*$ has \emph{ascent of finite building on $\cat{S}$} if whenever $M$ and $N$ are objects of $\cat{S}$ with $\vp_*(M)\models_R \vp_*(N)$ then $M\models_S N$. By a slight abuse of notation we will say $\vp$ has ascent of finite building on $\cat{S}$. 
\end{chunk}

As discussed in the introduction much of this work is motivated by the homotopical characterizations of surjective complete intersection maps in \cite{Briggs/Grifo/Pollitz:2021,Briggs/Iyengar/Letz/Pollitz:2021,Letz:2021,Pollitz:2019}. Before continuing to the main result of the article, we recall a method from \cite{Briggs/Iyengar/Letz/Pollitz:2021} used to detect the complete intersection property  among surjective maps of finite projective dimension.

\begin{chunk}\label{c:cicharacterization}
Let $K$ denote the Koszul complex on a minimal generating set for $\fm_S$ over $S$. If $\pdim_R S$ is finite, then there exist $M_1,\ldots,M_n$ in $\dflbcat{S}$ satisfying the following properties: \begin{enumerate}
	\item each $M_i$ finitely builds $K$ in $\dcat{R}$;
	\item if each $M_i$ finitely builds $K$ in $\dcat{S}$, then $\vp$ is complete intersection. 
\end{enumerate}	
Indeed explicitly one can take the $M_i$ as mapping cones in $\dcat{S}$ on a generating set for $\Ker \pi^2(\vp)$; see \cref{piexactsequence}.
The first point is justified by \cite[Lemma~2.5 \& Lemma~3.6]{Briggs/Iyengar/Letz/Pollitz:2021}
while the second condition holds by \cite[Theorem~3.12]{Briggs/Iyengar/Letz/Pollitz:2021}.
\end{chunk}

\begin{theorem}
\label{characterization2}
	A surjective map $\vp\colon R\to S$ of commutative noetherian local rings is exceptional complete intersection if and only if $\pdim_R S$ is finite and $\vp$ has ascent of finite building on $\dflbcat{S}$. 
\end{theorem}

\begin{proof} $(\Rightarrow)$:
As $\vp$ is exceptional complete intersection, \cref{nilpotencefinitelength} implies $\hc[2]{S}{R}$ acts nilpotently on each object of $\dflbcat{S}$. Now the desired result follows from \cref{NilpotencyAscentDescent}.	
 
$(\Leftarrow)$: Since $\pdim_R S$ is finite, there exist $M_1,\ldots,M_n$ in $\dflbcat{S}$ from \cref{c:cicharacterization}. Using the assumption that $\vp$ has ascent of building it follows from the two properties of the $M_i$ from \cref{c:cicharacterization} that $\vp$ is complete intersection. 

As $\pdim_R S$ is finite, $S\lotimes_R k\models_R k$ in $\dflbcat{R}$. Since $\vp$ has ascent of building on $\dflbcat{S}$, we conclude that
\[
S\lotimes_R k\builds[S] k\,.
\]
In particular $\hc[2]{S}{R}$ acts nilpotently on $k$, for $\hc[2]{S}{R}$ acts trivially on $S\lotimes_R k$. Thus $\vp$ is exceptional complete intersection, by \cref{characterization1}. 
\end{proof}

\begin{remark}
Instead of invoking \cref{characterization1} in the proof of $(\Leftarrow)$ of \cref{characterization2} one can argue as follows: Use the same argument as before that $\vp$ is complete intersection. Next factor $R\to E\to S$ where $E$ is the Koszul complex over $R$ on a minimal generating set for $\Ker\vp.$ There is the following commutative diagram obtained by restricting scalars
\begin{center}
\begin{tikzcd}
\dcat{R}& \dcat{E}\ar[l] \\
& \dcat{S}\ar[ul] \ar[u,swap,"\equiv"] 
\end{tikzcd} 
\end{center} where the vertical map is an equivalence since $E\xra{\simeq} S.$  Now as $\vp$ has ascent of building on $\dflbcat{S}$, so does $R\to E$ on $\dflbcat{E}$. 

The main point is that the Koszul complex over $R$ on a minimal generating set for $\fm_R$, denoted $K^R$, is an object of $\dflbcat{E}$ with the property \[E\otimes_R K^R\fbuilds[E] K^R\]
 if and only if $\vp$ is exceptional complete intersection; this can be verified by a direct calculation. Now as  $R\to E$ has ascent of building on $\dflbcat{E}$ and one always has the finite building \[E\otimes_R K^R\fbuilds[R] K^R\,,\] we obtain the desired conclusion from the previous remark. 
 
 Note that the argument sketched above justifies that when $\vp$ is complete intersection, $\vp$ is exceptional complete intersection if and only if $S\lotimes_E K^R$ has finite projective dimension over $S$.
 \end{remark}

\begin{chunk}
\label{dflasatriangulatedcat}
\cref{characterization2} is a characterization of exceptional complete intersection maps purely in terms of the structure of $\dcat R$ and $\dcat S$ as triangulated categories. Namely, for a triangulated category $\cat{T}$ one can consider its thick subcategory of compact objects, denoted $\cat{T^c}$. Define $\cat{T^{fl}_b}$ to be the full subcategory of $\cat{T}$ consisting of objects $X$ such that 
$\Hom_{\cat{T}}(C,X)$ is a finite length $\Hom_{\cat{T}}(C,C)$-module for each $C$ in $\cat{T^c}$. In the context of $\dcat{R}$, it is well known an $R$-complex belongs to $\dcat{R}^{\sf c}$ precisely when it has finite projective dimension over $R$, and $\dcat{R}_{\sf b}^{\sf fl}=\dflbcat{R}$.
\end{chunk}

Let $\Thick\cat{T}$ be the lattice of thick subcategories of a triangulated category $\cat{T}$.

\begin{chunk}
Let $\vp\colon R \to S$ be a surjective map of local rings, and  $\vp_*\colon \dcat{S}\to \dcat{R}$ the restriction functor.
For a thick subcategory $\cat{U}$ of $\dflbcat{R}$, we write $F(\cat{U})$ to denote the smallest thick subcategory containing all objects $M$ in $\dflbcat{S}$ such that $\vp_*(M)$ is in $\cat{U}$.  If $\cat{S}$ is a thick subcategory of $\dflbcat{S}$, then $G(\cat{S})$ is the smallest thick subcategory containing all objects $\vp_*(M)$ in $\dflbcat{R}$ such that $M$ is in $\cat{S}$. Thus one has maps of lattices 
\[
 \Thick \dflbcat{R}\stackrel[G]{F}{\rightleftarrows} \Thick\dflbcat{S}\,.
\]  
\end{chunk}

The following result is Theorem~\ref{introthm2}.

\begin{corollary}\label{cor:latticeiso}
	The map $\vp$ is exceptional complete intersection if and only if $\pdim_RS$ is finite and $F$ and $G$ are mutually inverse lattice isomorphisms. 
\end{corollary}

\begin{proof} 
$(\Leftarrow)$: This follows immediately from \cref{characterization2}.

$(\Rightarrow):$
From \cref{nilpotencefinitelength} and \cref{NilpotencyAscentDescent}, $FG$ is the identity on $\Thick \dflbcat{S}.$ 

Fix a thick subcategory $\cat{U}$ in $\Thick \dflbcat{R}$. To show $GF(\cat{U})=\cat{U}$, it suffices to show for each object $L$ in $\cat{U}$ 
one has
\begin{equation}
\label{thickcatequality}
\thick_R L=\thick_R \vp_*(M) \qquad \text{where $M=S\lotimes_R L$.} 
\end{equation}
Since $\pdim_R S$ is finite, $L\models_R \vp_*(M)$, by Lemma~\ref{le:Rbuilding}. Let $K$ be the Koszul complex on a minimal generating set for $\fm_R$ and  $E$ the Koszul complex on a minimal generating set for $\Ker(\vp)$. By \ref{ch:HN} one has that $E\models_R K$ and so
\[
\vp_*(M)\simeq E\otimes_R L\builds[R]K\otimes_R L\,.
\]
Furthermore, as $\hh(L)$ is finite length over $R$, it follows that $K\otimes_R L\models_R L$ and so $\vp_*(M)\models_R L$; thus, \cref{thickcatequality} holds. 
\end{proof}

We end this article with two applications; the first should be compared with the factorization theorems in \cite[5.7.1]{Avramov:1999} and \cite[Corollary~4.2]{Briggs/Iyengar/Letz/Pollitz:2021}

\begin{corollary}
	Consider surjective maps $R\xra{\vp}S\xra{\psi} T$ of commutative noetherian local rings. If $\vp$ and $\psi$ are exceptional complete intersection, then $\psi\vp$ is exceptional complete intersection. The converse holds when $\pdim_S T$ is finite.
\end{corollary}

\begin{proof} 
It is clear from the definition that the class of exceptional complete intersection maps is closed under composition. It also follows trivially from the characterization in \cref{characterization2}. 

Assume that $\pdim_S T$ is finite and $\psi\vp$ is exceptional complete intersection. It can be argued as in \cite[Corollary~4.2]{Briggs/Iyengar/Letz/Pollitz:2021} that $\pdim_R S$ is finite. Also, it is clear that $\psi$ has ascent of finite building on $\dflbcat{S}$ since $\psi\vp$ does. Hence, $\psi$ is exceptional complete intersection. Now as $\psi$ and $\psi\vp$ have ascent of finite building on finite length objects, we can apply \cref{cor:latticeiso} to deduce $\vp$ is exceptional complete intersection. 
\end{proof}

Another application of \cref{cor:latticeiso} is provided when $R$ is an equicharacteristic Cohen--Macaulay ring of \emph{minimal multiplicity}.  When $R$ is Cohen--Macaulay Abhyankar~\cite{Abhyankar:1967} proved the Hilbert--Samuel multiplicity of $R$ is bounded below by ${\rm codim} R+1$; when equality holds $R$ is said to be of minimal multiplicity. 

\begin{chunk}\label{c:minmult}
For the rest of the section, assume  $R$ is an equicharacteristic Cohen--Macaulay ring of minimal multiplicity and infinite residue field $k$. There exists an exceptional complete intersection map $R\to S$ such that $\fm_S^2=0$; cf. \cite[4.6.14]{Bruns/Herzog:1998}. As $S$ is equicharacteristic, it has the form $k[x_1,\ldots,x_n]/(x_1,\ldots,x_n)^2$ for some $n\geqslant 0.$

The structure of the lattice $\mathcal{T}=\Thick\dflbcat{R}$ is especially simply when $R$ is also Gorenstein; this forces  $n\leqslant 1$.

If $n=0$, then $R$ is regular and so $\mathcal{T}$ is simply $0 \subsetneq \dflbcat{R}$.

If $n=1,$ then the lattice $\mathcal{T}$ is 
\[
0\subsetneq \thick_R(K)\subsetneq \dflbcat{R}
\] where $K$ is a Koszul complex on any system of parameters for $R$. This is a consequence of \cref{cor:latticeiso} and \cite{Carlson/Iyengar:2015,Liu/Pollitz:2021}.
\end{chunk}
When $R$ is not Gorenstein, $\Thick\dflbcat{R}$ possesses a far wilder structure.

\begin{corollary}
\label{cor:artinian}
Let $R$ be an equicharacteristic Cohen--Macaulay non-Gorenstein ring of minimal multiplicity with infinite residue field $k$.  Then there exists an infinite descending binary tree of finitely generated thick subcategories in $\Thick \dflbcat{R}$. 
\end{corollary}

\begin{proof}
The assumptions on $R$ force $S$, from \cref{c:minmult}, to be  $k[x_1,\ldots,x_n]/(x_1,\ldots,x_n)^2$ with $n\geqslant 2$. Now the stated result is a consequence of \cref{cor:latticeiso} and the work of Elagin and Lunts~\cite[Theorem~A]{Elagin/Lunts:2022}.
\end{proof} 

\bibliographystyle{amsplain}
\bibliography{refs}

\end{document}